\documentclass[reqno]{amsart}
\usepackage{amsmath, pb-diagram}
\usepackage{amssymb}
\usepackage{amscd}

\begin{document}
\renewcommand\baselinestretch{1.3}
\title[On A Class of Lifting Modules]{On A Class of Lifting Modules}

\author{  Hatice Inankil}
\address{Department of Mathematics, Ankara University, 06100 Ankara, Turkey}
\email{hatinankil@gmail.com}

\author{Sait Hal\i c\i o\u glu}
\address{Department of Mathematics, Ankara University, 06100 Ankara, Turkey}
\email{halici@science.ankara.edu.tr}

\author{Abdullah Harmanci}
\address{Department of Mathematics, Hacettepe University, 06550 Ankara, Turkey}
\email{harmanci@hacettepe.edu.tr}
\date{}
\newtheorem{thm}{Theorem}[section]
\newtheorem{lem}[thm]{Lemma}
\newtheorem{prop}[thm]{Proposition}
\newtheorem{cor}[thm]{Corollary}
\newtheorem{df}[thm]{Definition}
\newtheorem{nota}{Notation}
\newtheorem{note}[thm]{Remark}
\newtheorem{ex}[thm]{Example}
\newtheorem{exs}[thm]{Examples}


\begin{abstract}
In this paper, we introduce principally $\delta$-lifting modules
which are analogous to  $\delta$-lifting modules and principally
$\delta$-semiperfect modules as a generalization of
$\delta$-semiperfect modules and investigate their properties.
\end{abstract}
\maketitle
\section{Introduction}
Throughout this paper all rings have an identity, all modules
considered are unital right modules.   Let $M$ be a module and
$N,P$ be submodules of $M$. We call $P$ {\it a supplement} of $N$
in $M$ if $M=P+N$ and $P\cap N$ is small in $P$. A module $M$ is
called {\it supplemented} if every submodule of $M$ has a
supplement in $M$. A module $M$ is called {\it lifting} if, for
all $N\leq M$, there exists a decomposition $M=A\oplus B$ such
that  $A\leq N$ and $N\cap B$ is small in $M$. Supplemented and
lifting modules have been discussed by several authors (see
\cite{CCNW,  MM, Os}) and these modules are useful in
characterizing semiperfect and right perfect rings (see \cite{MM,
Wi}).

In this note,  we study and investigate  principally $\delta$-lifting modules and principally $\delta$-semiperfect modules.
 A module $M$  is called {\it principally $\delta$-lifting} if for each cyclic submodule has the
 {\it  $\delta$-lifting property}, i.e.,  for each $m \in M$, $M$ has a decomposition
$M=A\oplus B$ with $A\leq mR$ and $mR\cap B$ is $\delta$-small in
$B$, where $B$ is called {\it a $\delta$-supplement} of $mR$.  A module $M$ is called  {\it principally
$\delta$-semiperfect} if, for each $m \in M$, $M/mR$ has a
projective $\delta$-cover.   We prove that if
$M_1$ is  semisimple, $M_2$ is  principally $\delta$-lifting,
$M_1$ and $M_2$ are relatively projective, then $M = M_1\oplus
M_2$ is a  principally $\delta$-lifting module. Among others we also prove that
for a principally $\delta$-semiperfect module $M$, $M$ is
principally $\delta$-supplemented, each  factor module of $M$ is
principally $\delta$-semiperfect, hence any homomorphic image and
any direct summand of $M$ is principally $\delta$-semiperfect.
As an application, for a projective module $M$, it is shown that
$M$ is principally $\delta$-semiperfect if and only if it is principally $\delta$-lifting, and therefore
 a ring $R$ is principally $\delta$-semiperfect if and only if it is principally $\delta$-lifting.

In section 2,  we give some properties of $\delta$-small submodules
that we use in the paper, and in section 3, principally
$\delta$-lifting modules are introduced and various properties of
principally $\delta$-lifting and $\delta$-supplemented modules are
obtained. In section 4,  principally $\delta$-semiperfect modules
are defined and characterized in terms of principally
$\delta$-lifting modules.

 In what follows, by $\Bbb Z$, $\Bbb Q$, $\Bbb Z_n$ and
$\Bbb Z/\Bbb Zn$ we denote, respectively, integers, rational
numbers, the ring of integers and the $\Bbb Z$-module of integers
modulo $n$. For unexplained concepts and
notations, we refer the reader to \cite{AF, MM}.

\section{$\delta$-Small Submodules}

Following Zhou \cite{Zh}, a submodule $N$ of a module $M$ is
called a {\it $\delta$-small submodule} if, whenever $M=N+X$ with
$M/X$ singular, we have $M=X$.  We begin by stating the next lemma which is contained in
\cite[Lemma 1.2 and 1.3]{Zh}.
\begin{lem}\label{ilk} Let $M$ be a module. Then we have the
following.
\begin{enumerate}
\item  If $N$ is $\delta$-small in $M$ and $M=X+N$, then $M=X\oplus Y$ for a projective semisimple submodule $Y$ with $Y\subseteq N$.
\item  If $K$ is $\delta$-small in $M$ and $f:M\rightarrow N$ is a
homomorphism, then $f(K)$ is $\delta$-small in $N$. In
particular, if $K$ is $\delta$-small in  $M\subseteq N$, then $K$
is $\delta$-small in $N$.
\item  Let $K_{1}\subseteq M_{1}\subseteq M$, $K_{2}\subseteq
M_{2}\subseteq M$ and $M=M_{1}\oplus M_{2}$. Then $K_{1}\oplus
K_{2}$ is $\delta$-small in $M_{1}\oplus M_{2}$ if and only if
$K_{1}$ is $\delta$-small in $M_{1}$ and $K_{2}$is $\delta$-small
in $M_{2}$.
\item  Let $N$, $K$ be submodules of $M$ with $K$ is $\delta$-small in
$M$ and $N\leq K$. Then $N$ is also $\delta$-small in $M$.
\end{enumerate}
\end{lem}
\begin{lem}\label{iki1} Let $M$ be a module and $m\in M$.  Then the following are
equivalent.
\begin{enumerate}
\item  $mR$ is not $\delta$-$small$ in $M$.
\item  There is a maximal submodule $N$ of $M$ such that $m \not\in N$ and $M/N $ singular.
\end{enumerate}
\end{lem}
\begin{proof}
$ (1)  \Rightarrow  (2)  $ Let $\Gamma :=\{B \leq M\mid B\neq M,
mR+B=M$, $M/B~{\rm singular}\}$. Since  $mR$ is not $\delta$-small
in $M$, there exists a proper submodule $B$ of $M$ such that $mR+B
= M$ and $M/B$ singular. So $\Gamma$ is non empty. Let $\Lambda$
be a nonempty totally ordered subset of $ \Gamma$ and $B_0:=
{\cup_{B\in    \Lambda } }B $. If $m$ is in $ B_0$ then there is a
$B\in \Lambda$ with $m \in B$. Then $B=mR+B=M$ which is a contraction.
So we have $m \notin B_0$ and $B_0\neq M$. Since $mR+B_0=M$ and
$M/B_0$ singular,  $B_0$ is upper bound in $\Gamma $. By Zorn's
Lemma, $\Gamma $ has a maximal element, say $N$. If $N$ is a
maximal submodule of $M$ there is nothing to do. Assume that there
exists a submodule $K$ containing $N$ properly. Since  $N$ is maximal in $\Gamma $, $K$ is not in
$\Gamma $. Since $M = mR + N$ and $N\leq K$, so $M = mR + K$.
$M/K$ is singular as a homomorphic image of singular module $M/N$.
Hence $K$ must belong to the $\Gamma$. This is the required
contradiction.

\noindent $(2)\Rightarrow (1)$  Let $N$ be a maximal submodule
with $m\in M\setminus N$ and $M/N$ singular. We have $M = mR+N$.
Then $mR$ is not $\delta$-small in $M$.
\end{proof}

Let $A$ and $B$ be submodules of $M$ with $A\leq B$. $A$ is called
a {\it $\delta$-cosmall submodule} of $B$ in $M$ if $B/A$ is
$\delta$-small in $M/A$. Let $A$ be a submodule of $M$. $A$ is
called a {\it $\delta$-coclosed submodule} in $M$ if $A$ has no
proper $\delta$-cosmall submodules in $M$. A submodule $A$ is called $\delta$-{\it coclosure}
of $B$ in $M$ if $A$ is $\delta$-coclosed submodule of $M$ and it
is $\delta$-cosmall submodule of $B$. Equivalently, for any
submodule $C\leq A$ with $A/C$ is $\delta$-small in $M/C$ implies
$C = A$ and $B/A$ is $\delta$-small in $M/A$. Note that
$\delta$-coclosed submodules need not always exist.

\begin{lem} Let $A$ and $B$ be submodules of $M$ with $A \leq B$. Then we have:
\begin{enumerate}
\item  $A$ is $\delta$-cosmall submodule of $B$ in $M$ if and only
if $M = A + L$ for any submodule $L$ of $M$ with $M = B + L$ and
$M/L$ singular.
\item  If $A$ is $\delta$-small and $B$ is $\delta$-coclosed in
$M$, then $A$ is $\delta$-small in $B$.
\end{enumerate}
\end{lem}
\begin{proof} (1) Necessity:   Let $M = B + L$ and $M/L$ be singular. We have  $M/A = B/A + (L+A)/A$ and $M/(L+A)$ is singular as homomorphic image of
singular module $M/L$. Since $B/A$ is $\delta$-small, $M/A =
(L+A)/A$ or $M = L + A$. \\
Sufficiency: Let $M/A = B/A + K/A$ and
$M/K$ singular. Then $M
= B + K$. By hypothesis,  $M = A + K$ and so  $M = K$. Hence $A$ is $\delta$-cosmall submodule of $B$ in $M$.\\
(2) Assume that $A$ is $\delta$-small submodule of  $M$ and $B$ is
$\delta$-coclosed in $M$. Let $B = A + K$ with $B/K$ singular.
Since $B$ is $\delta$-coclosed in $M$, to complete the proof, by
part (1) it suffices to show that $K$ is $\delta$-small submodule
of $B$ in M. Let $M = B + L$ with $M/L$ singular. By assumption,
$M = A + K + L = K + L$ since $M/(K+L)$ is singular. By (1), $K$
is $\delta$-small submodule of $B$ in $M$.
\end{proof}
\begin{lem} Let $A$,  $B$ and $C$ be submodules of $M$ with $M = A + C$ and $A\subseteq B$. If $B \cap C$ is a $\delta$-small submodule of $M$,
then $A$ is a $\delta$-cosmall submodule of $B$ in $M$.
\end{lem}
\begin{proof} Let $M/A = B/A + L/A$ with $M/L$ singular. We have $M = B + L$ and $B = A + (B\cap C)$. Then $M = A + (B\cap C) + L = (B\cap C) + L$.
Hence $M = L$ since $B\cap K$ is $\delta$-small in $M$ and $M/L$ is singular. Hence $B/A$ is $\delta$-small in $M/A$. Thus $A$ is $\delta$-cosmall
submodule of $B$ in $M$.
\end{proof}

\section{Principally $\delta$-Lifting Modules}
In this section,   we study and investigate some properties of
principally $\delta$-lifting modules.  The following  definition
is motivated by  \cite[Lemma 3.4]{Zh} and Lemma \ref{fini}.

\begin{df} {\rm A module $M$ is called {\it finitely  $\delta$-lifting}  if for any finitely generated submodule
$A$ of $M$ has the {\it $\delta$-lifting property}, that is,  there
is a decomposition $M=N\oplus S$ with $N\leq A$ and $A\cap S$ is
$\delta$-small in $S$. In this case $A\cap S$ is $\delta$-small in
$S$ if and only if $A\cap S$ is $\delta$-small in $M$. A module $M$  is called {\it principally $\delta$-lifting} if for each cyclic submodule has the
 {\it  principally $\delta$-lifting property}, i.e.,  for each $m \in M$, $M$ has a decomposition
$M=A\oplus B$ with $A\leq mR$ and $mR\cap B$ is $\delta$-small in
$B$.} \end{df}

\begin{ex}\label{semi} {\rm Every submodule of any semisimple module satisfies principally $\delta$-lifting
property.}
\end{ex}

\begin{ex}\label{p-n} {\rm Let  $p$ be a prime integer and $n$ any  positive  integer. Then the
$\Bbb Z$-module $M = \Bbb Z/\Bbb Zp^{n}$ is  a principally
$\delta$-lifting module.}
\end{ex}

\noindent Lemma \ref{fini} is proved in \cite{Wi} and \cite{Zh}.

\begin{lem}\label{fini} The following are equivalent for a module
$M$.
\begin{enumerate}
\item  $M$ is finitely $\delta$-lifting.
\item  $M$ is principally $\delta$-lifting.
\end{enumerate}
\end{lem}

Let $M$ be a module and $N$ a  submodule of $M$. A submodule
$L$ is called a {\it $\delta$-supplement} of $N$ in $M$ if $M = N
+ L$ and $N\cap L$ is $\delta$-small in $L$(therefore in $M$).
\begin{prop}\label{dik} Let $M$ be a principally $\delta$-lifting module. The we have:
\begin{enumerate}
\item  Every direct summand of $M$ is a principally
$\delta$-lifting module.
\item  Every cyclic submodule $C$ of $M$
has a $\delta$-supplement $S$ which is a direct summand, and $C$
contains a complementary summand of $S$ in $M$.
\end{enumerate}
\end{prop}
\begin{proof} (1) Let $K$ be a direct summand of $M$ and  $k\in K$. Then $M$ has a decomposition $M=N\oplus S$ with $N\leq kR$ and $kR\cap S$ is $\delta$-small in
$M$. It follows that $K=N\oplus (K\cap S)$, and $kR\cap (K\cap
S)\leq kR\cap
S$ is $\delta$-small in  $M$ and so $kR\cap (K\cap S)$ is $\delta$-small in  $K$. Therefore $%
K$ is a principally $\delta$-lifting module.\\
(2) Assume that $M$ is a principally $\delta$-lifting module and $C$ is a cyclic submodule of $M$. Then we have $M=N \oplus S$,
where $N \leq C$ and $C \cap S$ is $\delta$-small in $M$. Hence
$M=N+S \leq C+S \leq M$, we have $M= C+S$. Since  $S$ is direct
summand and $C \cap S$ is $\delta$-small in $M$, $C \cap S$ is
$\delta$-small in $S$. Therefore $S$ is a $\delta$-supplement of $C$
in $M$.
\end{proof}

\begin{thm}\label{guz}
The following are equivalent for a module $M$.
\begin{enumerate}
\item  $M$  is a principally $\delta$-lifting module.
\item  Every cyclic submodule $C$ of $M$ can be written as $C=N\oplus S$, where $N$ is direct summand and $S$ is $\delta$-small in $M$.
\item  For  every cyclic submodule $C$ of $M$,  there is a direct summand $A$ of $M$ with $A \leq C$ and $C/A$ is $\delta$-small in $M/A$.
\item  Every cyclic submodule $C$ of $M$ has a $\delta$-supplement $K$ in $M$ such that $C\cap K$ is a direct summand in $C$.
\item  For  every cyclic submodule $C$ of $M$,  there is an idempotent $e \in End(M)$ with $eM\leq C$ and $(1-e)C$ is $\delta$-small in $(1-e)M$.
\item  For each $m\in M$, there exist ideals $I$ and $J$ of $R$
such that $mR=mI\oplus mJ$, where $mI$ is direct summand of $M$
and $mJ$ is $\delta$-small in  $M$.
\end{enumerate}
\end{thm}
\begin{proof}
(1)$\Rightarrow $(2)  Let $C$ be a cyclic submodule of $M$. By
hypothesis  there exist $N$ and $S$ submodules of $M$ such that $N
\leq C$, $C \cap S$ is $\delta$-small in $M$ and $M=N \oplus S$.
Then we have $C=N \oplus (C \cap S)$.


\noindent
$(2)\Rightarrow (3)$  Let $C$ be a cyclic submodule of $M$ . By
hypothesis,  $C=N\oplus S$,  where $N$ is direct summand and $S$ is
$\delta$-small in $M$. Let $ \pi : M \rightarrow M/N $ be the
natural projection. Since $S$ is $\delta$-small in  $M$, we have
$\pi(S)$ is $\delta$-small in $M/N$.
Since $\pi (S)\cong S\cong C/N$,  $C/N$ is $\delta$-small in $M/N$.

\noindent
$(3)\Rightarrow (4)$  Let $C$ be a cyclic submodule of $M$.   By
hypothesis,  there is a direct summand $A \leq M$ with $A \leq C$
and $C/A$ is $\delta$-small in  $M/A$. Let $M=A \oplus A'$ . Hence
$C=A \oplus (A' \cap C)$.  Let $ \sigma  : M/A \rightarrow A'$
denote the obvious isomorphism.
Then $\sigma (C/A) = A'\cap C$ is $\delta$-small in $A'$.

\noindent
$(4)\Rightarrow (5)$  Let $C$ be any cyclic submodule of $M$ and
$K\leq M$ such that $C\cap K$ is direct summand of $C$, $M=C+K$
and $C\cap K$ is $\delta$-small in $K$ . So $C=(C \cap K) \oplus X
$ for some  $X\leq C$ . Then  $M=X + (C \cap K) + K=X \oplus K$ .
Let $e:M\rightarrow X ~;~ e(x+k)=x$ and $(1-e):M\rightarrow
K~;~e(x+k)=k$ are projection maps. $e(M) \leq  X  \leq C$ and
$(1-e)C=C \cap (1-e)M=C \cap K$ is $\delta$-small in
$(1-e)M$.

\noindent $(5)\Rightarrow (6)$  Let $mR$ be any cyclic submodule
of $M$ . By hypothesis,  there exists an idempotent $e \in End(M)$
such that  $eM\leq mR$, $M=eM\oplus (1-e)M$ and $(1-e)mR$ is $\delta$-small in $(1-e)M$. Note that $(mR)\cap
((1-e)M)$  = $(1-e)mR$ ( for if $m = em_1 + y$,  where $em_1\in eM$,
$y\in (mR)\cap ((1-e)M)$. Then $(1-e)m = em_1 + (1-e)y = y$ and so
$(1-e)mR\leq (mR)\cap ((1-e)M)$. Let $mr = (1-e)m'\in (mR)\cap
((1-e)M)$. Then $mr = (1-e)mr\in (1-e)mR$. So $(mR)\cap
((1-e)M)\leq (1-e)mR$. Thus $(mR)\cap ((1-e)M) = (1-e)mR$ ). So
$mR=eM\oplus (1-e)mR$. Let $I=\{r \in R : mr \in eM \}$ and $J=\{t
\in R : mt \in (1-e)mR \}$. Then $mR=mI\oplus mJ$ , $mI=eM$ and
$mJ=(1-e)mR$ is $\delta$-small in $(1-e)M$.

\noindent
 $(6)\Rightarrow (1)$  Let $m\in M$.   By hypothesis, there exist  ideals $I$ and $J$ of $R$ such
that $mR=mI\oplus mJ$, where $mI$ is direct summand and $mJ$ is
$\delta$-small in $M$. Let $M=mI\oplus K$ for some submodule $K$.
Since $K \cap mR \cong mJ$ and  $mJ$ is $\delta$-small in $M$ , M
is principally $\delta$-lifting.
\end{proof}
Note that every lifting module is principally $\delta$-lifting. There are
principally $\delta$-lifting modules but not lifting.

\begin{ex}\label{örn2}{\rm Let $M$ be the $\mathbb Z$-module $ \mathbb Q$ and  $m\in M$.
It is well known that every cyclic submodule $mR$ of $M$ is small,
therefore $\delta$-small in $M$. Hence $M$ is a principally
$\delta$-lifting $\mathbb Z$-module. If  $N$ is a nonsmall proper
submodule of $M$, then $N$ is neither direct summand nor contains
a direct summand of $M$. It follows that $M$ is not a lifting
$\mathbb Z$-module.}
\end{ex}

It is clear that  every $\delta$-lifting module is principally
$\delta$-lifting. However the converse is not true.

\begin{ex}\label{zit}{\rm Let $R$ and $T$ denote the rings in \cite[Example 4.1]{Zh}, where
 \begin{center}$R=\underset{i=1}{\overset{\infty }{\sum_{}^{}  }} \bigoplus \Bbb{Z}_2 + \Bbb{Z}_2.1=\{(f_1, f_2, \dots , f_n,f, f,\dots)
\in\underset{i=1}{\overset{\infty }{\prod }} \Bbb{Z}_2\} $\end{center}  and $T= \left\{ \left[
\begin{array}{rr}
x&y\\
o&x\\
\end{array}
\right] ~:~x \in R,~y \in Soc(R)  \right\}.$ Then ${\rm Rad}_\delta (T)= \left[
\begin{array}{rr}
0&Soc(R)\\
0&0\\
\end{array}
\right]$ and $T/ {\rm Rad}_\delta (T)$ is not semisimple as isomorphic to $R$. So $T$ is not $\delta$-semiperfect by \cite[Theorem 3.6]{Zh}. Hence
$T$ is not a $\delta$-lifting module over $T$. It is easy to show
 that $T/ {\rm Rad}_\delta (T)$
lift to idempotents of $T$, so $T$ is a semiregular ring. Since $T$ is a $\delta$-semiregular ring, every finitely generated right ideal $H$ of $T$
can be written as $H=aT\oplus S$, where $a^2=a \in T$ and $S \leq {\rm Rad}_\delta (T) $ by  \cite[Theorem 3.5]{Zh}. Hence $T$ is  a principally
$\delta$-lifting module.}
\end{ex}

\begin{prop} Let $M$ be a principally $\delta$-lifting module. If $M = M_1 + M_2$ such that $M_1\cap M_2$ is cyclic, then $M_2$ contains a
$\delta$-supplement of $M_1$ in $M$.
\end{prop}
\begin{proof} Assume that $M = M_1 + M_2$ and $M_1\cap M_2$ is cyclic. Then we have $M_1\cap M_2 = N \oplus S$,
where $N$ is direct summand of $M$ and
$S$ is $\delta$-small in $M$. Let $M = N\oplus N'$ and $M_2 =
N\oplus (M_2\cap N')$. It follows that $M_1\cap M_2 = N \oplus
(M_1\cap M_2 \cap N') = N\oplus S$. Let $\pi : M_2 = N\oplus
(M_2\cap N')\rightarrow N'$ be the natural projection. It follows
that $\pi(M_1\cap M_2\cap N') = M_1\cap M_2\cap N' = \pi(S)$.
Since $S$ is $\delta$-small in $M$, it  is $\delta$-small in
$N'$ by Lemma \ref{ilk}. Hence $M = M_1 + (M_2\cap N')$, $M_2\cap
N'\leq M_2$ and $M_1\cap (M_2\cap N')$ is $\delta$-small in $M_2\cap N'$. $M_2\cap N'$ is contained in $M_2$ and
a $\delta$-supplement of $M_1$ in $M_2$. This completes the proof.
\end{proof}

Let $M$ be a module. A submodule $N$ is called {\it fully
invariant} if for each endomorphism $f$ of $M$, $f( N )\leq N$.
 Let $S=End(M_R)$, the
ring of $R$-endomorphisms of $M$. Then $M$ is a left $S$-, right
$R$-bimodule and a principal submodule $N$ of the right $R$-module $M$ is
fully invariant if and only if $N$ is a sub-bimodule of $M$.
Clearly $0$ and $M$ are fully invariant submodules of $M$. The
right $R$-module $M$ is called a {\em duo module} provided every
submodule of $M$ is  fully invariant. For the readers' convenience
we state and prove Lemma \ref{kes} which is proved in \cite{AHP}.

\begin{lem} \label{kes} Let a module $M=\underset{i\in I}\bigoplus M_i$ be a direct sum of submodules
$M_i$ $(i\in I)$  and  let $N$ be a fully invariant submodule of
$M$. Then $N=\underset{i\in I}\bigoplus (N\cap M_i)$.
\end{lem}
\begin{proof} For each $j\in I$, let $p_j: M\rightarrow M_j$ denote the canonical projection and let $i_j: M_j \rightarrow M$ denote inclusion.
Then $i_jp_j$ is an endomorphism of $M$ and hence
$i_jp_j(N)\subseteq N$ for each $j\in I$. It follows that $N
\subseteq \underset{j \in I}\bigoplus  \,\,i_jp_j(N)\subseteq   \underset{j \in I}\bigoplus (N\cap M_j)\subseteq N$,
so that $N= \underset{j \in I}\bigoplus (N\cap
M_j)$.
\end{proof}

One may suspect  that if $M_1$ and $M_2$ are principally  $\delta$-lifting modules, then  $M_1\oplus M_2$ is also principally  $\delta$-lifting.
But this is not the case.

\begin{ex}\label{ilkör} {\rm Consider the $\Bbb Z$-modules $M_1 = \Bbb Z/\Bbb Z2$ and $M_2 = \Bbb Z/\Bbb Z8$. It is clear that $M_1$ and $M_2$ are
principally $\delta$-lifting. Let $M = M_1\oplus M_2$. Then $M$ is
not a principally $\delta$-lifting $\Bbb Z$-module. Let $N_1 = (\overline{1}, \overline{2})\Bbb Z $ and
$N_2 =(\overline{1}, \overline{1}) \Bbb Z$.  Then $M = N_1 + N_2$, $N_1$ is not a direct summand of $M$ and
does not contain any nonzero direct summand of $M$. For any proper
submodule $N$ of $M$, $M/N$ is singular $\Bbb Z$-module. Hence the
principal submodule does not satisfy $\delta$-lifting property. It
follows that $M$ is not principally $\delta$-lifting $\Bbb
Z$-module. By the same reasoning, for any prime integer $p$, the
$\Bbb Z$-module $M = (\Bbb Z/\Bbb Zp)\oplus (\Bbb Z/\Bbb Zp^3)$ is
not principally $\delta$-lifting. }
\end{ex}

We have already observed by the preceding example that the direct
sum of principally $\delta$-lifting modules need not be
principally $\delta$-lifting. Note the following fact.

\begin{prop} Let $M = M_1\oplus M_2$ be a decomposition of $M$ with $M_1$  and $M_2$ principally $\delta$-lifting modules. If $M$ is a duo module,
then $M$ is principally $\delta$-lifting.
\end{prop}
\begin{proof}Let $M = M_1\oplus M_2$ be a duo module and  $mR$ be a submodule of  $M$. By Lemma \ref{kes},
$mR= ((mR) \cap M_1) \oplus ((mR) \cap M_2)$. Since $ (mR) \cap M_1$ and $(mR) \cap M_2$ are principal submodules of $M_1$ and $M_2$ respectively,
there exist $A_1, B_1 \leq M_1$ such that
$A_1 \leq (mR) \cap M_1 \leq M_1=A_1 \oplus B_1$,  $B_1 \cap ((mR) \cap
M_1)=B_1 \cap (mR)$ is $\delta$-small in $B_1$, and   $A_2, B_2 \leq
M_2$ such that $A_2 \leq (mR) \cap M_2 \leq M_2=A_2 \oplus B_2$,
$B_2 \cap ((mR) \cap M_2)=B_2 \cap (mR)$ is $\delta$-small in $B_2$. Then
$M=A_1 \oplus A_2 \oplus B_1 \oplus B_2$, $A_1 \oplus A_2 \leq N$
and $(mR) \cap ( B_1 \oplus B_2)= ((mR) \cap B_1) \oplus ((mR) \cap B_2)$  is
$\delta$-small  in $M_1 \oplus M_2$.
\end{proof}

\begin{lem}\label{proj} The following are equivalent for a module $M = M'\oplus
M''$.
\begin{enumerate}
\item  $M'$ is $M''$-projective.
\item  For each submodule $N$ of $M$ with $M = N + M''$, there
exists a submodule $N'\leq N$ such that $M = N'\oplus M''$.
\end{enumerate}
\end{lem}
\begin{proof} See \cite[41.14]{Wi}
\end{proof}
\begin{thm} Let $M_1$ be a semisimple module and $M_2$ a principally $\delta$-lifting module. Assume that $M_1$ and $M_2$ are relatively projective.
Then $M = M_1\oplus M_2$ is principally $\delta$-lifting.
\end{thm}
\begin{proof} Let $0\neq m\in M$ and let $K=M_1 \cap ((mR)+M_2)$. We divide the proof into two
cases:\\
\noindent {\it Case (i):} $K\neq 0$. Then $M_1=K\oplus K_1$ for some
submodule $K_1$ of $M_1$ and so $M=K\oplus K_1 \oplus
M_2=(mR)+(M_2\oplus K_1)$. Hence $K$ is $M_2\oplus K_1$-projective. \linebreak By Lemma \ref{proj}, there exists a submodule $N$ of
$mR$ such that $M=N\oplus (M_2\oplus K_1)$. We may assume
$(mR)\cap(M_2\oplus K_1)\neq 0$. Note that for any submodule $L$
of $M_2$,  we have $(mR)\cap(L+K_1) = L\cap((mR)+K_1)$. In
particular \linebreak $(mR)\cap(M_2+K_1) = M_2\cap(mR+K_1)$. Then $mR =
N\oplus (mR)\cap (K_1\oplus M_2)$. There exist $n\in N$ and $m'\in
(mR)\cap (K_1\oplus M_2)$ such that $m = n + m'$. Then $nR = N$
and $m'R = (mR)\cap (K_1\oplus M_2)$. Since $(mR)\cap(M_2+K_1) =
M_2\cap((mR)+K_1)$, $M_2\cap((mR)+K_1)$ is a principal submodule
of $M_2$ and $M_2$ is principally $\delta$-lifting, there exists a
submodule $X$ of $M_2\cap((mR)+K_1)= (mR)\cap(M_2\oplus K_1)$ such
that $M_2=X\oplus Y$ and $Y\cap M_2\cap((mR)+K_1)=
Y\cap((mR)+K_1)$ is $\delta$-small in $M_2\cap((mR)+K_1)$ and in
$M_2$. Hence $M=(N\oplus X)\oplus (Y\oplus K_1 )$. Since $N\oplus
X\leq mR$ and $(mR)\cap(Y\oplus K_1)=Y\cap((mR)+K_1)$,
$(mR)\cap(Y\oplus K_1)=Y\cap ((mR)+K_1)$ is
$\delta$-small in $Y\oplus K_1$. So $M$ is $\delta$-lifting. \\
\noindent {\it  Case (ii): }  $K=0$. Then $mR\leq M_2$. Since $M_2$ is
$\delta$-lifting, there exists a submodule $X$ of $mR$ such that
$M_2=X\oplus Y$ and $(mR)\cap Y$ is $\delta$-small in $Y$ for some
submodule $Y$ of $M_2$. Hence $M=X\oplus (M_1 \oplus Y)$. Since
$(mR)\cap (M_1 \oplus Y) = (mR)\cap Y$ and $(mR)\cap (M_1 \oplus
Y) = (mR)\cap Y$ is $\delta$-small in $Y$. By Lemma \ref{ilk} (3),
$(mR)\cap (M_1\oplus Y)$ is $\delta$-small in $M_1\oplus Y$. It
follows that $M$ is $\delta$-lifting.
\end{proof}
A module  $M$ is said to be {\it a principally semisimple} if every cyclic submodule is a direct summand of $M$.
Tuganbayev calls a principally semisimple module as a regular
module in \cite{Cl}. Every semisimple module is principally
semisimple.  Every principally semisimple module is principally
$\delta$-lifting. For a module $M$, we write Rad$_{\delta}(M) = \sum \{L \mid L$ is a $\delta$-small submodule
of $M\}$.

\begin{lem}\label{ss} Let $M$ be a  principally $\delta$-lifting module. Then $M/${\rm Rad}$_{\delta}(M)$ is a principally semisimple module.
\end{lem}
\begin{proof} Let $m\in M$. There exists $M_1\leq mR$ such that $M = M_1\oplus M_2$ and
$(mR)\cap M_2$ is $\delta$-small in $M_2$. So$(mR)\cap M_2$ is
$\delta$-small in $M$.  Then
$$M/ {\rm Rad}_{\delta}(M) = [(mR+
{\rm Rad}_{\delta}(M))/ {\rm Rad}_{\delta}(M)] \oplus [(M_2 + {\rm
Rad}_{\delta}(M))/ {\rm Rad}_{\delta}(M)]$$  because $(mR +$ {\rm
Rad}$_{\delta}(M)) \cap (M_2 + $Rad$_{\delta}(M)) =
$Rad$_{\delta}(M)$. Hence every principal submodule of
$M/$Rad$_{\delta}(M)$ is a direct summand.
\end{proof}

\begin{prop} Let $M$ be a principally $\delta$-lifting module. Then $M=M_{1}\oplus M_{2}$,
where $M_{1}$ is a principally semisimple module and $M_{2}$ is a
module with {\rm Rad}$_{\delta}(M)$ essential in $M_{2}$.
\end{prop}
\begin{proof} Let $M_{1}$ be a submodule of $M$ such that Rad$_{\delta}(M)\oplus M_{1}$ is essential in $M$ and  $m\in M_1$. Since $M$ is
principally $\delta$-lifting, there exists a direct summand
$M_{2}$ of $M$ such that $M_2\leq mR$, $M = M_2 \oplus M'_{2}$ and
$mR\cap M'_{2}$ is $\delta$-small in $M$. Hence $mR\cap M'_{2}$ is
a submodule of Rad$_{\delta}(M)$ and so $mR\cap M'_{2} = 0$. Then
$m\in M_2$ and $mR = M_2$. Since $M_2 \cap$ Rad$_{\delta}(M) = 0$,
$M_2$ is isomorphic to a submodule of $M/$Rad$_{\delta}(M)$. By
Lemma \ref{ss},  $M/$Rad$_{\delta}(M)$ is principally semisimple,
$M_2$ is principally semisimple. On the other hand,
Rad$_{\delta}(M) = $Rad$_{\delta}(M'_2)$ is essential in $M_2$
that it is clear from the construction of $M'_2$.
\end{proof}

 A nonzero
module $M$ is called {\it $\delta$-hollow} if every proper
submodule is $\delta$-small in $M$, and $M$ is {\it principally
$\delta$-hollow} if every proper cyclic submodule is
$\delta$-small in $M$, and $M$ is {\it finitely $\delta$-hollow}
if every proper finitely generated submodule is $\delta$-small in
$M$. Since finite direct sum of $\delta$-small submodules is
$\delta$-small, $M$ is principally $\delta$-hollow if and only if
it is finitely $\delta$-hollow.

\begin{lem} The following are equivalent for an indecomposable module
$M$.
\begin{enumerate}
\item  $M$ is a principally $\delta$-lifting module.
\item $M$ is a principally $\delta$-hollow module.
\end{enumerate}
\end{lem}

\begin{proof} (1)$\Rightarrow $(2) Let $m\in M$.  Since $M$  is a principally $\delta$-lifting module,
there exist $N$ and $S$ submodules of  $M$   such that $N
\leq mR$,  $mR \cap S$ is $\delta$-small in  $M$ and $M=N \oplus
S$. By hypothesis,  $N=0$ and $S=M$. So
that $ mR \cap S =mR$ is $\delta$-small in  $M$.\\
(2)$\Rightarrow $(1) Let $m\in M$.  Then  $mR = (mR) \oplus (0)$. By (2) $mR$ is $\delta$-small and
$(0)$ is direct summand in $M$. Hence $M$ is a principally $\delta$-lifting module.
\end{proof}

\begin{lem}\label{uclu} Let $M$ be a module, then we have
\begin{enumerate}
\item  If $M$ is principally $\delta$-hollow, then every factor module is principally $\delta$-hollow.
\item  If $K$ is $\delta$-small submodule  of $M$ and $M/K$ is principally $\delta$-hollow, then $M$ is principally $\delta$-hollow.
\item  $M$ is principally $\delta$-hollow if and only if $M$ is
local or {\rm Rad}$_{\delta}(M) = M$.
\end{enumerate}
\end{lem}
\begin{proof} (1) Assume that $M$ is principally $\delta$-hollow and $N$  a submodule of $M$. Let $m + N\in M/N$ and $(mR + N)/N + K/N = M/N$.
Suppose
that $M/K$ is singular. We have $mR + K = M$. Since $M/K$ is singular  and $M$ is principally $\delta$-hollow, $M = K$.\\
(2) Let $m\in M$. Assume that $mR + N = M$ for some submodule $N$
with $M/N$ singular. Then $(m+K)R = (mR + K)/K$ is a cyclic
submodule of $M/K$ and $(mR + K)/K + (N + K)/K = M/K$ and $M/(N +
K)$ is singular as an homomorphic image of $M/N$. Hence $(N + K)/K
= M/K$ or $N + K = M$. By hypothesis $N =
M$.\\
(3) Suppose that $M$ is  principally $\delta$-hollow and it is not
local. Let $N$ and $K$ be two distinct maximal submodules of $M$
and $k\in K\setminus N$. Then $M = kR + N$ and $M/N$ is a simple
module, and so  $M/N$ is a singular or projective module. If $M/N$ is
singular, then $M = N$ since $kR$ is $\delta$-small. But this is
not possible since $N$ is maximal. So $M/N$ is projective. Hence
$N$ is direct summand. So $M = N\oplus N'$ for some nonzero submodule $N'$ of $M$, that is,  $N$ and $kR$ are proper submodules of $M$.
Since every proper
submodule of $M$ is contained in Rad$_{\delta}(M)$, $M$ = {\rm
Rad}$_{\delta}(M)$. The converse is clear.
\end{proof}
\begin{prop}\label{iki} Let $M$ be a module. Then the following are equivalent.
\begin{enumerate}
\item   $M$ is principally $\delta$-hollow.
\item If  $N$ is submodule with $M/N$ cyclic, then $N$ is a
$\delta$-small submodule of $M$.
\end{enumerate}
\end{prop}
\begin{proof} (1) $\Rightarrow$ (2) Assume that $N$ is a submodule with $M/N$ cyclic. Lemma \ref{ilk}  implies that $M/N$ is principally
$\delta$-hollow  since being $\delta$-small is preserved under
homomorphisms.  Since $M/N$ has maximal submodules, and by Lemma
\ref{uclu},  $M/N$ is local. There exists a unique maximal submodule $N_1$ containing $N$. Hence $N$ is small,  therefore it is $\delta$-small.\\
(2) $\Rightarrow$ (1) We prove that every cyclic submodule is
$\delta$-small in $M$. So let $m\in M$ and $M = mR + N$ with
$M/N$ singular. Then $M/N$ is cyclic. By hypothesis,   $N$ is $\delta$-small submodule of $M$.  By Lemma \ref{ilk},  there exists a projective
semisimple submodule $Y$ of $N$ such that $M = (mR)\oplus Y$.  Let  $Y = \bigoplus\limits_{i\in I} N_i$ where each $N_i$  is simple. Now
we write $M = ((mR) \bigoplus \limits_{i\neq  j} N_j)\oplus N_i$. Then  $M / ((mR) \bigoplus\limits_{i\neq  j} N_j)$ is cyclic module as it is isomorphic to simple module $N_i$.
By hypothesis,    $((mR) \bigoplus\limits_{i\neq  j} N_j)$ is $\delta$-small in $M$.  Again by Lemma \ref{ilk},  there exists a projective semisimple submodule $Z$ of
 $((mR) \bigoplus\limits_{i\neq  j} N_j)$ such that $M = Z\oplus N_i$.  Hence $M$ is  projective semisimple module.  So $M = N\oplus N'$ for some
submodule $N'$. Then $N'$ is projective.  $M/N$ is projective as it is isomorphic to $N'$.  Hence $M/N$ is both singular and projective module. Thus
$M = N$.
\end{proof}
\section{Applications}
In this section,   we introduce and study some properties of principally $\delta$-semiperfect modules. By \cite{Zh}, a projective module $P$ is
called {\it a projective $\delta$-cover } of a module $M$ if there exists an epimorphism $f: P\longrightarrow M$ with $Kerf$ is $\delta$-small in
$P$, and a ring is called {\it $\delta$-perfect} (or {\it $\delta$-semiperfect}) if every $R$-module (or every simple $R$-module) has a  projective
$\delta$-cover. For more detailed discussion on $\delta$-small submodules, $\delta$-perfect and $\delta$-semiperfect rings, we refer to \cite{Zh}. A
module $M$ is called {\it principally $\delta$-semiperfect} if every factor module of $M$ by a cyclic submodule has a projective $\delta$-cover. A
ring $R$ is called {\it principally $\delta$-semiperfect} in case the right $R$-module $R$ is principally $\delta$-semiperfect. Every $\delta$-semiperfect module is
 principally $\delta$-semiperfect. In \cite{Zh},  a ring
$R$ is called {\it $\delta$-semiregular} if every cyclically presented R-module has a projective $\delta$-cover.
\begin{thm}\label{belli} Let $M$ be a projective module. Then the following are equivalent.
\begin{enumerate}
\item $M$ is principally $\delta$-semiperfect.
\item $M$ is principally $\delta$-lifting.
\end{enumerate}
\end{thm}
\begin{proof} {\rm (1)}$\Rightarrow$ {\rm (2)} Let $m\in M$ and $P  \stackrel{f}\rightarrow M/mR$ be a projective $\delta$-cover and \linebreak
$ M  \stackrel{\pi}\rightarrow M/mR$ the natural epimorphism.

\[\begin{diagram}
\node[2]{M}\arrow{sw,t,..}g\arrow{s,r}\pi \\
\node{P}\arrow{e,t}f\node{M/mR}\arrow{e}\node{0}
\end{diagram}\]
Then there exists a map $M  \stackrel{g}\rightarrow P$ such that
$fg = \pi$. Then $P = g(M) + $Ker$(f)$. Since Ker$(f)$ is
$\delta$-small, by Lemma \ref{ilk}, there exists a projective
semisimple submodule $Y$ of Ker$(f)$ such that $P = g(M)\oplus Y$.
So $g(M)$ is projective. Hence $M = K\oplus Ker(g)$ for some
submodule $K$ of $M$. It is easy to see that $g(K\cap mR) =
g(K)\cap $Ker$(f)$ and Ker$(g)\leq mR$. Hence $M = K + mR$. Next
we prove $K\cap (mR)$ is $\delta$-small in $K$. Since Ker$(f)$ is
$\delta$-small in $P$, $g(K)\cap $Ker$(f)= g(K\cap mR)$ is
$\delta$-small in $P$ by Lemma \ref{ilk} (4). Hence $K\cap (mR)$
is $\delta$-small in $K$ since $g^{-1}$ is an isomorphism from
$g(M)$ onto $K$.\\
{\rm (2)}$\Rightarrow$ {\rm (1)} Assume that $M$ is a principally
$\delta$-lifting module. Let $m\in M$. There exist direct summands
$N$ and $K$  of $M$ such that $M = N\oplus K$, $N\leq mR$ and
$mR\cap K$ is $\delta$-small in $K$. Let
$K\stackrel{\pi}\rightarrow M/mR$ denote the natural epimorphism
defined by $\pi(k) = k + mR$ where $k\in K$, $k + mR\in M/mR$. It
is obvious that \linebreak Ker$(\pi)$ = $mR\cap K$. It follows that $K$ is
projective $\delta$-cover of $M/mR$. So $M$ is principally
$\delta$-semiperfect.
\end{proof}

\begin{cor}  Let $R$ be a ring. Then the following are equivalent.
\begin{enumerate}\item $R$ is principally $\delta$-semiperfect.
\item $R$ is principally $\delta$-lifting.
\item $R$ is $\delta$-semiregular.
\end{enumerate}
\end{cor}
\begin{proof}  (1) $\Leftrightarrow$ (2)  Clear by Theorem \ref{belli}. \\(2) $\Leftrightarrow$ (3)  By Theorem \ref{guz} (2),   $R$ is principally $\delta$-lifting if and
only if for every principal right ideal $I$ of $R$ can be written as $I = N\oplus S$, where $N$ is direct summand and $S$ is $\delta$-small in $R$.
This is equivalent to being $R$ $\delta$-semiregular since for any ring $R$, Rad$_{\delta}(R)$ is $\delta$-small in $R$ and each submodule of a
$\delta$-small submodule is $\delta$-small.
\end{proof}

The module $M$ is called {\it principally $\delta$-supplemented} if
every cyclic submodule of $M$ has a $\delta$-supplement in $M$.
Clearly, every $\delta$-supplemented module is principally
$\delta$-supplemented. Every
principally $\delta$-lifting module is principally
$\delta$-supplemented.   In a subsequent paper we investigate
principally $\delta$-supplemented modules in detail. Now we prove:

\begin{thm}\label{suppl}  Let $M$ be a principally $\delta$-semiperfect module. Then
\begin{enumerate}
\item  $M$ is principally $\delta$-supplemented.
\item  Each  factor module of $M$ is principally
$\delta$-semiperfect, hence any homomorphic image and any direct
summand of $M$ is principally $\delta$-semiperfect.
\end{enumerate}
\end{thm}
\begin{proof} (1) Let $m\in M$. Then $M/mR$ has a projective $\delta$-cover $P\stackrel{\beta}\rightarrow M/mR$. There exists
$P \stackrel{\alpha}\rightarrow M$ such that the following diagram is
commutative, $\beta =\pi \alpha$,  \linebreak where $M\stackrel{\pi}\rightarrow
M/mR$ is the natural epimorphism.

\[\begin{diagram}
\node[2]{P}\arrow{sw,t,..}{\alpha}\arrow{s,r}\beta \\
\node{M}\arrow{e,t}{\pi}\node{M/mR}\arrow{e}\node{0}
\end{diagram}\]

\noindent Then $M = \alpha(P) + mR$, and $\alpha(P)\cap mR$ is
$\delta$-small in $\alpha(P)$, by Lemma \ref{ilk} (1).
Hence $M$ is principally $\delta$-supplemented.\\
(2) Let $M \stackrel{f}\rightarrow N$ be an epimorphism and $nR$
a cyclic submodule of $N$.  Let \linebreak $m\in f^{-1}(nR)$ and $P
 \stackrel{g}\rightarrow M/(mR)$ be a projective $\delta$-cover. Define
$M/(mR) \stackrel{h} \rightarrow N/nR$ by $h(m' + mR) = f(m') + nR$,  where
 $m' + mR\in M/(mR)$. Then Ker$(g)$ is contained in Ker$(hg)$.
 By projectivity of $P$,  there is a map $\alpha$ from $P$ to $N$ such that $hg = \pi\alpha$.
\[\begin{diagram}
\node{P}\arrow{e,t}{g}\arrow{s,r,..}{\alpha}\node{M/mR}\arrow{s,r}{h}\\
\node{N}\arrow{e,t}{\pi}\node{N/nR}\arrow{e}\node{0}
\end{diagram}\]
It is routine to check that $(nR)\cap \alpha(P) =
\alpha($Ker$(g))$. By Lemma \ref{ilk} (2),  $\alpha(Ker(g))$ is
$\delta$-small in $N$ since Ker$(g)$ is $\delta$-small. Let
$x\in$Ker$(\pi\alpha)$. Then $hg(x) = (\pi\alpha)(x) = 0$ or
$\alpha(x)\in (nR)\cap \alpha(P)$. So Ker$(\pi\alpha)$ is
$\delta$-small. Hence $P$ is a projective $\delta$-cover for
$N/(nR)$.
\end{proof}

\begin{thm} Let $P$ be a projective module with {\rm Rad}$_{\delta} (P)$ is $\delta$-small in $P$. Then the following are
equivalent.
\begin{enumerate}
\item $P$  is principally $\delta$-lifting.
\item $P / ${\rm Rad}$_\delta (P)$ is principally semisimple and, for
any cyclic submodule $\overline{x}R$ of $P / ${\rm Rad}$_\delta
(P)$ that is a direct summand of $P/ ${\rm Rad}$_\delta (P)$,
there exists a cyclic direct summand $A$ of $P$ such that
$\overline{x}R=\overline{A}$.
\end{enumerate}
\end{thm}
\begin{proof} (1)$\Rightarrow$ (2)  Since  $P$ is a principally $\delta$-lifting module,   $P / ${\rm Rad}$_\delta (P)$ is principally semisimple by
 Lemma \ref{ss}.  Let $\overline{x}R$ be any cyclic submodule of
$P / $ {\rm Rad}$_\delta (P)$.  By Theorem \ref{guz},  there
exists a direct summand $A$ of $P$ and a $\delta$-small submodule
$B$ such that $xR = A\oplus B$. Since $B$ is contained in {\rm
Rad}$_{\delta}(R)$, $xR + $ {\rm Rad}$_{\delta}(R) = A + $ {\rm
Rad}$_{\delta}(R)$. Hence $\overline{x}R = \overline{A}$.

\noindent (2)$\Rightarrow$ (1) Let $xR$ be any cyclic submodule of
$P$. Then we have $P / $ {\rm Rad}$_\delta (P) = [(xR + ${\rm
Rad}$_\delta (P))/ ${\rm Rad}$_\delta (P)] \oplus [U/ $ {\rm
Rad}$_\delta (P)]$ for some $U\leq P$. By (2),  there exists a
direct summand $A$ of $P$ such that $P = A\oplus B$ and $U = B + $
{\rm Rad}$_\delta (P)$. Then $P = A \oplus B = A + U + $ {\rm
Rad}$_\delta (P)$. Since {\rm Rad}$_\delta (P)$ is $\delta$-small
in $P$, there exists a projective and semisimple submodule $Y$ of
$P$ such that $P = A \oplus B = (A + U)\oplus Y$. Since $P$ is
projective, $A + B$ is also projective and so by Lemma \ref{proj},
we have $A + B = V\oplus B$ for some $V\leq A$. Hence $P = V\oplus
B \oplus Y$. On the other hand $(xR)\cap (B \oplus Y) = (xR)\cap B
\leq (xR)\cap U \leq$ {\rm Rad}$_{\delta}(R)$. Since {\rm
Rad}$_{\delta}(R)$ is $\delta$-small in $P$, it is  $\delta$-small
in $B \oplus Y$ by Lemma \ref{ilk} (3). Thus $P$ is principally
$\delta$-lifting.
\end{proof}

\end{document}